\documentclass[twoside,11pt]{amsart}
\usepackage{a4wide}
\usepackage[utf8]{inputenc}
\usepackage[english]{babel}
\usepackage[fixlanguage]{babelbib}
\usepackage{color,graphicx}
\usepackage{amssymb,amsfonts,amsmath,amsthm,amscd, epsfig}
\usepackage{microtype}
\usepackage[all]{xy}

\title{Inverse Semigroupoid Actions and Representations}
\author[Lautenschlaeger and Tamusiunas]{Wesley G. Lautenschlaeger and Thaísa Tamusiunas}
\address{Departamento de Matemática Pura e Aplicada - Instituto de Matemática e Estatística, Universidade Federal do Rio Grande do Sul \\  Av. Bento Gon\c{c}alves, 9500, 91509-900. Porto Alegre-RS, Brazil}
\email{wesley\_gl@hotmail.com, thaisa.tamusiunas@gmail.com}
\date{}

\usepackage{graphicx}
\usepackage{amssymb}
\usepackage{amsmath}
\usepackage{amsthm}
\usepackage{indentfirst}
\usepackage{tikz}
\usetikzlibrary{cd}

\numberwithin{contador}{section}

\newtheorem{teo}{Theorem}[section]

\newtheorem{propo}[teo]{Proposition}
\newtheorem{lema}[teo]{Lemma}
\newtheorem{defi}[teo]{Definition}
\newtheorem{obs}[teo]{Remark}
\newtheorem{exe}[teo]{Example}

\begin{document}

\begin{abstract}
We show that there is a one-to-one correspondence between the partial actions of a groupoid $G$ on a set $X$ and the inverse semigroupoid actions of the Exel's inverse semigroupoid $S(G)$ on $X$. We also define inverse semigroupoid representations on a Hilbert space $H$, as well as the Exel's partial groupoid $C^*$-algebra $C_p^*(G)$, and we prove that there is a one-to-one correspondence between partial groupoid representations of $G$ on $H$, inverse semigroupoid representations of $S(G)$ on $H$ and $C^*$-algebra representations of $C_p^*(G)$ on $H$. 
\end{abstract}

\maketitle

\section{Introduction}
The notion of partial action on a $C^*$-algebra was introduced by Ruy Exel first time in \cite{exel1994}. The interesting thing about this study is that it allowed to classify several important classes of $C^*$-algebras as crossed products by partial actions. Michael Dokuchaev and Exel proved in \cite{dokuchaev2005associativity} the associativity of crossed products by partial actions, provided that the ring is semiprime,  and they used crossed products to study relations between partial actions of groups on algebras and partial representations. In the same work, they studied the enveloping actions and they gave a condition to a partial group action on an unital ring to be globalizable, that is,  the partial action is a restriction of a global one. As it was mentionated by Dokuchaev in \cite{dokuchaevsurvey}, the relationship between partial isomorphisms and global ones is relevant in several branches of mathematics, such as operator theory, topology, logic, graph theory, differential geometry, group theory and, mainly, as the interest of this paper, the theory of semigroups and its generalizations.

Recentely, many applications of groupoids to the study of partial actions have been presented, as well as inverse semigroup actions on topological groupoids. For instance, in \cite{buss2017}, the authors constructed saturated Fell bundles over inverse semigroups and non-Hausdorff étale groupoids and interpreted these as actions on $C^*$-algebras by Hilbert bimodules and describe the section algebras of these Fell bundles. In \cite{abadie2004} it was proved that, as in the case of global actions, any partial action gives rise to a groupoid provided with a Haar system, whose $C^*$-algebra agrees with the crossed product by the partial action. Furthermore, ring theoretic results of global and partial actions of groupoids were obtained in \cite{bagio2011partial, bagio2012partial, bagio2016partial, bagio2017partial, corta, renault, renault2}.

In \cite{exel1998partial} Exel constructed an inverse semigroup $S(G)$ associated to a group $G$ and showed that the actions of $S(G)$ are in one-to-one correspondence with the partial actions of $G$.Further, Kellendonk and Lawson showed in \cite{kellendonk2004partial} that the inverse semigroup $S(G)$ is isomorphic to the Birget-Rhodes expansion $\tilde{G}^{\mathcal{R}}$ of the group $G$. In \cite{lawson2006expansions}, Lawson, Margolis and Steinberg generalized this expansion for the case where $G$ is an inverse semigroup.

In the same work above, Exel defines partial representations of $G$ on a Hilbert space $H$ and representations of $S(G)$ on $H$, and presented a one-to-one correspondence between these two and representations of a $C^*$-algebra $C_p^*(G)$ (given by a certain crossed product) on $H$. Both structures - the inverse semigroup $S(G)$ and the $C^*$-algebra $C_p^*(G)$ - depend exclusively of the group $G$. Later, in \cite{exel2010actions}, Exel and Vieira presented a definition of crossed product for actions of inverse semigroups on $C^*$-algebras and proved that the crossed product by a partial action of a group is isomorphic to the crossed product by a related action of $S(G)$. 

In a groupoid context, W. Cortes \cite{cortes} conjectured the correspondence between partial groupoid actions and actions of the inverse semigroupoid associated and gave the first step in the direction of a broader theory. In this paper, we will make the construction of the Exel's semigroupoid using generators and relations, but this is exactly the Birget-Rhodes expansion defined by Gilbert in \cite{gilbert2005actions} for the case of ordered groupoids.

The purpose of this paper is to analyse the one-to-one relations worked by Exel in \cite{exel1998partial} and Exel and Vieira in \cite{exel2010actions} in a groupoid and inverse semigroupoid context. The paper is organized as follows. In Section 2, we define the Exel's semigroupoid $S(G)$ and we observe that this structure is the Birget-Rhodes expansion of $G$. In Section 3 we characterize partial groupoid actions on a set and, given a groupoid $G$ and a set $X$, we show that there is a one-to-one correspondence between the partial actions of $G$ on $X$ and the inverse semigroupoid actions of $S(G)$ on $X$. In Section 4, we define the algebraic crossed product of an inverse semigroupoid on a semiprime algebra $R$ by an action $\beta$ and we demonstrate that the algebraic crossed product of $R$ by an action $\beta$ on $S(G)$ is isomorphic to the the algebraic crossed product of $R$ by an action $\alpha$ on $G$, where $\alpha$ and $\beta$ are intrinsically related. In Section 5, we define inverse semigroupoid representations and partial groupoid representations on a Hilbert space $H$. Proposition 5.3 states that there is a one-to-one correspondence between  partial groupoid representations of $G$ on $H$ and inverse semigroupoid representations of $S(G)$ on $H$. Finally, in Section 6 we define the Exel's partial groupoid $C^*$-algebra $C_p^*(G)$ and we prove that exists a one-to-one correspondence between partial groupoid representations of $G$ on $H$, inverse semigroupoid representations of $S(G)$ on $H$ and $C^*$-algebra representations of $C_p^*(G)$ on $H$. 

\section{Inverse Semigroupoid Arising from a Groupoid}

Let $S$ be a non-empty set equipped with a binary operation partially defined, denoted by concatenation. Given $s,t \in S$, we write $\exists st$ when the product $st$ is defined. The set $S$ is called a \textit{semigroupoid} if the associativity holds when it makes sense, that is, 
\begin{enumerate}
    \item[(i)] for all $r,s,t \in S$, $\exists (rs)t$ if and only if $\exists r(st)$, and in this case $(rs)t = r(st)$; and
    \item[(ii)] for all $r,s,t \in S$, $\exists rs$ and $\exists st$ if and only if $\exists (rs)t$.
\end{enumerate}

We say that $S$ is an \textit{inverse semigroupoid} if for all $s \in S$ exists a unique element $s^{-1} \in S$ such that $\exists ss^{-1}$, $\exists s^{-1}s$ and the equalities 
\begin{align*}
    ss^{-1}s = s; \quad \quad s^{-1}ss^{-1} = s^{-1}
\end{align*}
hold. We say that an element $e \in S$ is an \emph{idempotent} if $\exists ee$ and $e^2 = e$. Notice that in this case $e = e^{-1}$. We denote by $E(S)$ the set of idempotents in $S$. One can prove that a regular semigroupoid $S$ (every element has a non-necessarily unique inverse) is an inverse semigroupoid if and only if the idempotent elements of $S$ commute \cite[Lemma 3.3.1]{liu2016free}.

It is immediate that every groupoid is an inverse semigroupoid, but the reciprocal is not necessarily true. In fact, if $S, T$ are inverse semigroups that are not groups, then $S \dot{\cup} T$, the disjoint union of $S$ and $T$, is an inverse semigroupoid that is not a groupoid. Thereby, inverse semigroupoid theory generalizes both inverse semigroup theory and groupoid theory. An \emph{inverse category} is an inverse semigroupoid that is also a category. If $C$ is an inverse category, we will denote by $C_0$ its set of identities, that is, $C_0 = \{e \in C : \text{ if } \exists ex \text{ (resp. } \exists ye\text{)} \text{ for } x \in C \text{ (resp. } y \in C\text{)}, \text{ then } ex = x \text{ (resp. } ey = y \text{)}\}$.

Let $S$ and $S'$ be semigroupoids. A map $f : S \rightarrow S'$ is called a \textit{semigroupoid homomorphism} if for all $s, t \in S$ such that $\exists st$, we have that $\exists f(s)f(t)$ and, in this case, $f(st) = f(s)f(t)$. The map is called \textit{semigroupoid anti-homomorphism} if for all $s, t \in S$ such that $\exists st$, then $\exists f(t)f(s)$ and, in this case, $f(st) = f(t)f(s)$.

A \textit{free semigroupoid} $S$ is a semigroupoid such that every element of $S$ is a finite concatenation of a subset of $S$, when that concatenation makes sense.

Our goal in this section is to extend to semigroupoids the construction of a specific semigroup constructed from a group $G$, given by Exel in \cite{exel1998partial}, called \textit{Exel's semigroup of} $G$.

Recall that a \emph{groupoid} is an inverse category $G$ in where $E(G) = G_0$. In this case, given $g \in G$, we denote by $r(g) = gg^{-1} \in G_0$ and $d(g) = g^{-1}g \in G_0$.
 
\begin{defi}
Let $G$ be a groupoid. For all $s \in G$, take the symbol $[s]$. We define the \textit{Exel's semigroupoid associated with the groupoid} $G$, denoted by $S(G)$, as the free semigroupoid generated by the symbols $[s]$ for each $s \in G$. We have that $\exists [s][t]$ in $S(G)$ if and only if $\exists st$ in $G$. Furthermore, the following relations hold:
\begin{enumerate}

    \item[(i)] if $s,t \in G$ and $\exists st$, then $[s^{-1}][s][t] = [s^{-1}][st]$;
    
    \item[(ii)] if $s,t \in G$ and $\exists st$, then $[s][t][t^{-1}] = [st][t^{-1}]$; and
    
    \item[(iii)] $[r(s)][s] = [s] = [s][d(s)]$.
    
\end{enumerate}

\end{defi}

\begin{obs}
\label{obs1}
Note that for all $t \in G$, exists $ t^{-1} \in G$ such that $\exists tt^{-1}$, $\exists t^{-1}t$, $tt^{-1} = r(t)$ and $t^{-1}t = d(t)$. Hence always $\exists [t][t^{-1}]$ and $\exists [t^{-1}][t]$ in $S(G)$. 
\end{obs}

The elements in the form $[t][t^{-1}]$ play an important rule in the characterization of $S(G)$. So let $\epsilon_t = [t][t^{-1}]$ in $S(G)$. We will prove that the $\epsilon$'s are idempotents is $S(G)$ and then we will use it to show that $S(G)$ is an inverse semigroupoid. 

The next proposition follows directly from the universal property of free semigroupoids \cite[Theorem 3.1.1]{liu2016free}.

\begin{propo}
\label{prop1}Given a groupoid $G$, a semigroupoid $S$ and a map $f : G \rightarrow S$, where:

\begin{enumerate}
    \item[\emph{(i)}] if $\exists st$ then $\exists f(s)f(t)$;

    \item[\emph{(ii)}] if $\exists st$ then $f(s^{-1})f(st) = f(s^{-1})f(s)f(t)$;
    
    \item[\emph{(iii)}] if $\exists st$ then $f(st)f(t^{-1}) = f(s)f(t)f(t^{-1})$; and
    
    \item[\emph{(iv)}] $f(r(s))f(s) = f(s) = f(s)f(d(s))$,
\end{enumerate}

\noindent then there is a unique semigroupoid homomorphism $\overline{f} : S(G) \rightarrow S$ such that $\overline{f}([t]) = f(t)$.
\end{propo}

Let $S(G)^{op}$ be the opposite semigroupoid of $S(G)$. That is, $S(G)^{op}$ coincides with $S(G)$ except by the partially defined multiplication $\cdot$, which is given by $$[s] \cdot [t] := [t][s],$$for all $[s], [t] \in S(G)$. Notice that $\exists [s] \cdot [t]$ if and only if $\exists ts$, for all $s, t \in G$.

\begin{propo}
There is an involutive anti-automorphism $^* : S(G) \rightarrow S(G)$ such that $[t]^* = [t^{-1}]$.
\end{propo}

\begin{proof}
Consider $f: G \rightarrow S(G)^{op}$, where $f(t) = [t^{-1}]$. 

We are going to show that $f$ holds the properties of Proposition \ref{prop1}. Indeed, given $(s,t) \in G^2$, we have that  $\begin{array}{rl}
     \exists st &  \Leftrightarrow \exists t^{-1}s^{-1} \Leftrightarrow \exists [t^{-1}][s^{-1}] \Leftrightarrow \exists [s^{-1}] \cdot [t^{-1}] \Leftrightarrow \exists f(s) \cdot f(t). 
    \end{array}$

This shows (i). To see that (ii) holds, observe that 

\begin{center}$
\begin{array}{rl}
     f(s^{-1}) \cdot f(st) &  = [s] \cdot [(st)^{-1}]  = [s] \cdot [t^{-1}s^{-1}] = [t^{-1}s^{-1}][s]  \\
     &  = [t^{-1}][s^{-1}][s]  = [s] \cdot [s^{-1}] \cdot [t^{-1}] = f(s^{-1}) \cdot f(s) \cdot f(t).
    \end{array}$
\end{center}

The proof of (iii) is similar. To see that (iv) holds, just note that

\begin{minipage}[b]{0.5\linewidth}
\begin{center}$
\begin{array}{rl}
     f(s) \cdot f(d(s)) &  = [s^{-1}] \cdot [d(s)^{-1}] \text{\quad \quad and} \\
      & = [s^{-1}] \cdot [d(s)]  \\
      & = [d(s)][s^{-1}]  \\
      & = [r(s^{-1})][s^{-1}]  \\
      & = [s^{-1}] = f(s) 
 \end{array}$
\end{center}
\end{minipage}
\begin{minipage}[l]{0.3\linewidth}
\begin{align*}
     f(r(s)) \cdot f(s) & = [r(s)^{-1}] \cdot [s^{-1}] \\
     & = [r(s)] \cdot [s^{-1}] \\
     & = [s^{-1}][r(s)] \\
     & = [s^{-1}][d(s^{-1})] \\
     & = [s^{-1}] = f(s).
\end{align*}
\end{minipage}

Therefore, there is a unique homomorphism $\overline{f} : S(G) \rightarrow S(G)^{op}$ such that $\overline{f}([t]) = f(t) = [t^{-1}]$. Define $^* : S(G) \rightarrow S(G)$ as $[t]^* = \overline{f}([t])$. We have $([t]^*)^* = \overline{f}([t])^* = [t^{-1}]^* = \overline{f}([t^{-1}]) = [t]$.

Moreover, if $\exists st$ then\begin{align*}
     ([s][t])^* &  = \overline{f}([s][t]) = \overline{f}([s]) \cdot \overline{f}([t]) = [s^{-1}] \cdot [t^{-1}] \\
     & = [t^{-1}][s^{-1}] = \overline{f}([t])\overline{f}([s]) = [t]^*[s]^*.
\end{align*}

Now it only remains for us to show that $^*$ is bijective. The surjectivity follows from the fact that we defined $^*$ on the generators of $S(G)$. The injectivity holds because the inverse map $s \mapsto s^{-1}$ is one-to-one in $G$. That completes the proof.
\end{proof} 

The above proposition enables us to characterize the elements $\epsilon_t$'s in $S(G)$. Remember by Remark \ref{obs1} that $\exists [t][t^{-1}]$, for all $t \in G$.

\begin{propo}\label{propepsilon}
Let $\epsilon_t = [t][t^{-1}] \in S(G)$. 
\begin{enumerate}
    \item[\emph{(i)}] $\epsilon_t^2 = \epsilon_t = \epsilon_t^*$;
    
    \item[\emph{(ii)}] if $\exists t^{-1}s$, then $\epsilon_t\epsilon_s = \epsilon_s\epsilon_t$;
    
    \item[\emph{(iii)}] if $\exists ts$, then  $[t]\epsilon_{s} = \epsilon_{ts}[t]$; and
    
    \item[\emph{(iv)}] if $\exists t^2$, then $[t]^2 = \epsilon_t[t^2]$.
\end{enumerate}

\end{propo}

\begin{proof}
With a little adaptation to the groupoid case, the proof has the same arguments as in \cite[Proposition 2.4]{exel1998partial}, so we will omit it.
\end{proof}

\begin{obs}
\label{obs2} Notice that $[t] = [t][d(t)] = [t][t^{-1}t] = [t][t^{-1}][t] = \epsilon_t[t]$, for all $[t] \in S(G)$.

\end{obs}

We can extend the above proposition to the case of the finite product of $\epsilon$'s. Given a groupoid $G$, define $X_g = \{ h \in G : r(h) = r(g) \} \subseteq G$. Note that $h \in X_g$ if and only if $\exists h^{-1}g$. It is obvious that $X_g = X_{r(g)}$, $h \in X_g \Rightarrow X_g = X_h$, and $G = \displaystyle \bigcup_{e \in G_0}^{\cdot} X_e$. 

The next result states that every element in $S(G)$ has a particular decomposition in terms of $\epsilon$'s.

\begin{propo}
\label{propdec}
Every element $\alpha \in S(G)$ admits a decomposition $\alpha = \epsilon_{r_1}\epsilon_{r_2} \cdots \epsilon_{r_n}[s]$, where $n \geq 0$, $r_1, \ldots , r_n, s \in X_s$. Moreover, 
\begin{enumerate}
    \item[\emph{(i)}] $r_i \neq r_j$ if $i \neq j$; and
    \item[\emph{(ii)}] $r_i \neq s$, $r_i \notin G_0$, $\forall i = 1, \ldots , n$.
\end{enumerate}
We say that this is a \emph{standard form} of the element $\alpha \in S(G)$.
\end{propo}

\begin{proof}
Let $S \subseteq S(G)$ be the set of elements of $S(G)$ that admit a decomposition in the standard form. Since we can have $n = 0$, it follows that $[g] \in S$, for all $g \in G$. Therefore, as all generators of $S(G)$ are in $S$, it only remains for us to prove that $S$ is a right ideal of $S(G)$. 

Take $\alpha \in S$. Then there are $r_1, \ldots , r_n, s \in X_s$ such that (i) and (ii) hold. Let also $t \in G$ be such that $\exists \alpha [t]$, that is, such that $\exists st$. Thereby, 

\begin{center}
$\begin{array}{r}
    \alpha [t] = \epsilon_{r_1} \cdots \epsilon_{r_n} [s] [t]  = \epsilon_{r_1} \cdots \epsilon_{r_n} [s] [s^{-1}] [s] [t]  = \epsilon_{r_1} \cdots \epsilon_{r_n} [s] [s^{-1}] [st] = \epsilon_{r_1} \cdots \epsilon_{r_n} \epsilon_s [st].
\end{array}$
\end{center}

If any $r_i = st$, it is enough to commute $\epsilon_{st}$ to the left $[st]$ and use the Remark \ref{obs2}. Then $\epsilon_{st}[st] = [st]$ into the decomposition. Furthermore, if $s \in G_0$, the element $\epsilon_s$ is ommited from the decomposition and $[st] = [t]$. Thus $S(G) = S$.
\end{proof}

\begin{propo}
For all $\alpha \in S(G)$, we have $\alpha \alpha^* \alpha = \alpha$ and $\alpha^* \alpha \alpha^* = \alpha^*$.
\end{propo}

\begin{proof}
Given $\alpha \in S(G)$, we have that $\alpha = \epsilon_{r_1} \cdots \epsilon_{r_n}[s]$ for certain $r_1, \ldots , r_n, s \in X_s$, satisfying (i) and (ii) of Proposition \ref{propdec}.

Observe that $^*$ is an anti-automorphism and every $\epsilon_{r_i}$ is self-adjoint, thus

\begin{center}
    $\alpha^* = (\epsilon_{r_1} \cdots \epsilon_{r_n}[s])^* = [s^{-1}]\epsilon_{r_n} \cdots \epsilon_{r_1}$.
\end{center}

Since $\exists [t][t^{-1}]$ and $\exists [t^{-1}][t]$ for all $t \in G$, then $\exists \alpha\alpha^*$ and $\exists \alpha^*\alpha$. Then, with a little adaptation to the groupoid case, the proof follows as \cite[Proposition 2.7]{exel1998partial}.
\end{proof}

To prove that the Exel's semigroupoid $S(G)$ is an inverse semigroupoid it only remains to show that if $\alpha \in S(G)$, then $\alpha^*$ is the only inverse of $\alpha$ in $S(G)$. For this, we define two special maps in $S(G)$.

Consider the identity map in $G$. Since every groupoid is a semigroupoid it is easy to see that this map holds the conditions (i)-(iv) of Proposition \ref{prop1}. Therefore we can extend it to a semigroupoid homomorphism $\partial : S(G) \rightarrow G$ where $\partial([g]) = g$, for all $g \in G$. Notice that given $\alpha = \epsilon_{r_1}\epsilon_{r_2} \cdots \epsilon_{r_n}[s] \in S(G)$ we have
\begin{align*}
     \partial(\alpha) & = \partial(\epsilon_{r_1}\epsilon_{r_2} \cdots \epsilon_{r_n}[s]) = \partial([r_1][r_1^{-1}] \cdots [r_n][r_n^{-1}][s]) = \partial([r_1])\partial([r_1^{-1}]) \cdots \partial([r_n])\partial([r_n^{-1}])\partial([s])  \\
     & = r_1r_1^{-1} \cdots r_nr_n^{-1}s = r(r_1)r(r_2) \cdots r(r_n)s = r(s)s =  s.
\end{align*}

Observe that we could ``split" the products by the homomorphism $\partial$ since $\exists [r_i^{-1}][r_{i+1}] \Leftrightarrow \exists r_i^{-1}r_{i+1}$ and $\exists [r_n^{-1}][s] \Leftrightarrow \exists r_n^{-1}s$. 

%that are mainteined by Proposition \ref{propdec}.

Define $\mathcal{P}_r(G)$ as the set formed by finite subsets $E \subseteq G$ such that $g \in E \Rightarrow r(g) \in E$. Let $\mathcal{F}(\mathcal{P}_r(G))$ be the semigroupoid of the functions of  $\mathcal{P}_r(G)$ into $\mathcal{P}_r(G)$. Consider the elements of $\mathcal{F}(\mathcal{P}_r(G))$ of the form $\phi_g: \mathcal{P}_r(G) \rightarrow \mathcal{P}_r(G)$ such that
\begin{center}
$\phi_g(E) = 
     \begin{cases}
       gE \cup \{ g, r(g) \}, &\quad\text{if } \exists gh, \forall h \in E \\
       \emptyset, &\quad\text{if exists } h \in E \text{ such that } \nexists gh,
     \end{cases}    $
\end{center}   
for all $g \in G$. Note that $\phi_g$ is well defined since that if $\exists gh$ then $d(g) = r(h)$, $gr(h) = gd(g) = g$ and $r(gh) = r(g)$. Hence we can define the map $\lambda : G \rightarrow \mathcal{F}(\mathcal{P}_r(G))$ given by $\lambda(g) = \phi_g$. It is straightforward to see that this map satisfies (i)-(iv) of Propostion \ref{prop1}. Therefore there is a unique semigroupoid homomorphism $\Lambda : S(G) \rightarrow \mathcal{F}(\mathcal{P}_r(G))$ where $\Lambda([g]) = \phi_g$. 

Given $\alpha = \epsilon_{r_1}\epsilon_{r_2} \cdots \epsilon_{r_n}[s] \in S(G)$ as in Propostion \ref{propdec} we have

\begin{center}$
    \begin{array}{rl}
     & \Lambda(\alpha)(\{ d(s) \}) = \Lambda(\epsilon_{r_1}\epsilon_{r_2} \cdots \epsilon_{r_n}[s])(\{ d(s) \}) \\
     & = \Lambda([r_1])\Lambda([r_1^{-1}]) \cdots \Lambda([r_n])\Lambda([r_n^{-1}])\Lambda([s])(\{ d(s) \}) \\
      & = \Lambda([r_1])\Lambda([r_1^{-1}]) \cdots \Lambda([r_n])\Lambda([r_n^{-1}])(\{ s, r(s) \}) \\
      & = \Lambda([r_1])\Lambda([r_1^{-1}]) \cdots \Lambda([r_n])(\{ r_n^{-1}s, r_n^{-1}r(s), r_n^{-1}, r(r_n^{-1}) \}) \\
      & =  \Lambda([r_1])\Lambda([r_1^{-1}]) \cdots \Lambda([r_n])(\{ r_n^{-1}s, r_n^{-1}, r(r_n^{-1}) \}) \\
      & = \Lambda([r_1])\Lambda([r_1^{-1}]) \cdots \Lambda([r_{n-1}^{-1}])(\{ r_nr_n^{-1}s, r_nr_n^{-1}, r_nr(r_n^{-1}), r_n, r(r_n) \}) \\
      & = \Lambda([r_1])\Lambda([r_1^{-1}]) \cdots \Lambda([r_{n-1}^{-1}])(\{ s, r_n, r(s) \}) \\
      & = \cdots = \{ s, r_1, \ldots , r_n, r(s) \}.
    \end{array}$
    \end{center}

This writing is well defined since $r_i \in X_s$ implies that $r(r_i) = r(s)$, for all $i = 1, \ldots, n - 1$.

We use $\partial$ and $\Lambda$ to prove the uniqueness of the decomposition of elements in $S(G)$ in the standard form.

\begin{propo}
The writing of elements in $S(G)$ in the standard form is unique up to the order of the $\epsilon$'s.
\end{propo}

\begin{proof}
Let $\alpha \in S(G)$ be such that $\alpha = \epsilon_{r_1} \cdots \epsilon_{r_n} [s]$, where $r_1, \ldots , r_n, s \in X_s$ as in Proposition \ref{propdec}. Assume that $\epsilon_{l_1} \cdots \epsilon_{l_m} [t]$ is another decomposition of $\alpha$, with $l_1 , \ldots , l_m, t \in X_t$ holding the same conditions. Thus
\begin{equation}
\label{eq1}
s = \partial(\epsilon_{r_1} \cdots \epsilon_{r_n} [s]) = \partial(\alpha) = \partial(\epsilon_{l_1} \cdots \epsilon_{l_m} [t]) = t,    
\end{equation}

Therefore $s = t$, which implies $\epsilon_{l_1} \cdots \epsilon_{l_m} [t] = \epsilon_{l_1} \cdots \epsilon_{l_m} [s]$. On the other hand,
\begin{equation}
\label{eq2}
\begin{split}
\{ r_1, \ldots , r_n, s, r(s) \} = \Lambda(\epsilon_{r_1} \cdots \epsilon_{r_n} [s])(\{ d(s) \}) = \Lambda(\alpha) \\ = \Lambda(\epsilon_{l_1} \cdots \epsilon_{l_m} [s])(\{ d(s)\}) = \{ l_1 \ldots , l_m, s, r(s) \}.
\end{split}
\end{equation}

Hence,
\begin{align*}
    \{ r_1, \ldots , r_n \} = \{ r_1, \ldots , r_n, s, r(s) \} \setminus \{ s, r(s) \} =  \{ l_1 \ldots , l_m, s, r(s) \} \setminus \{ s, r(s) \} = \{ l_1 \ldots , l_m\}. 
   \end{align*}

That concludes the proof, since the $\epsilon$'s commute and the conditions (i) and (ii) together with (\ref{eq1}) and (\ref{eq2}) guarantee that $m = n$.
\end{proof}

\begin{teo}\label{inverse}
$S(G)$ is an inverse semigroupoid.
\end{teo}

\begin{proof}
Let $\alpha = \epsilon_{r_1} \cdots \epsilon_{r_n} [s] \in S(G)$. We already know that $\alpha^* = [s^{-1}]\epsilon_{r_n} \cdots \epsilon_{r_1}$ is such that $\exists \alpha \alpha^*$, $\exists \alpha^* \alpha$, $\alpha \alpha^* \alpha = \alpha$ and $\alpha^* \alpha \alpha^* = \alpha^*$.

Assume that exists $\beta \in S(G)$ such that $\exists \alpha \beta$, $\exists \beta \alpha$, $\alpha \beta \alpha = \alpha$ and $\beta \alpha \beta = \beta$. We write $\beta^*$ in the standard form $\beta^* = \epsilon_{l_1} \cdots \epsilon_{l_m}[t]$. Thus, $\beta = [t^{-1}]\epsilon_{l_m} \cdots \epsilon_{l_1}$. Since $\exists \alpha \beta$ and $\exists \beta \alpha$, then $\exists st^{-1}$ and $\exists l_1^{-1}r_1$. Also, $s = \partial(\alpha) = \partial(\alpha \beta \alpha) = st^{-1}s$.

Note that $\exists t^{-1}s$, since $\exists t^{-1}l_m$, $\exists l_m^{-1}l_{m-1}$, $\ldots$, $\exists l_2^{-1}l_1$, $\exists l_1^{-1}r_1$, $\exists r_1^{-1}r_2$, $\ldots$, $\exists r_n^{-1}s$, from where $\exists t^{-1}l_ml_m^{-1}l_{m-1} \cdots l_2^{-1}l_1l_1^{-1}r_1r_1^{-1}r_2 \cdots r_n^{-1}s$ and this product is equal to $t^{-1}s$. Hence $\exists st^{-1}s$. Besides that, the uniqueness of the inverse element in $G$ tell us that $t = s$. So $\beta = [s^{-1}]\epsilon_{l_m} \cdots \epsilon_{l_1}$.

Observe that \begin{align*}
     \alpha\beta\alpha & = \epsilon_{r_1} \cdots \epsilon_{r_n} [s] [s^{-1}] \epsilon_{l_m} \cdots \epsilon_{l_1} \epsilon_{r_1} \cdots \epsilon_{r_n} [s] = \epsilon_{r_1} \cdots \epsilon_{r_n} \epsilon_s \epsilon_{l_m} \cdots \epsilon_{l_1} \epsilon_{r_1} \cdots \epsilon_{r_n} [s] \\
     &  = \epsilon_{r_1} \cdots \epsilon_{r_n} \epsilon_{l_m} \cdots \epsilon_{l_1} \epsilon_{r_1} \cdots \epsilon_{r_n} \epsilon_s [s] = \epsilon_{r_1} \cdots \epsilon_{r_n} \epsilon_{l_1} \cdots \epsilon_{l_m} [s]  = \epsilon_{r_1} \cdots \epsilon_{r_n} [s] = \alpha.
\end{align*}

Therefore, by the uniqueness of the writing of elements in $S(G)$ in the standard form, we have $\{ r_1, \ldots , r_n \} \cup \{ l_1, \ldots , l_m \} = \{ r_1, \ldots , r_n \}$, thus, $\{ l_1, \ldots , l_m \} \subseteq \{ r_1, \ldots , r_n \}$.

For the reverse inclusion note that
\begin{center}$
\begin{array}{rcl}
     \beta\alpha\beta  &  = & [s^{-1}]\epsilon_{l_m} \cdots \epsilon_{l_1} \epsilon_{r_1} \cdots \epsilon_{r_n} [s] [s^{-1}] \epsilon_{l_m} \cdots \epsilon_{l_1} = [s^{-1}]\epsilon_{l_1} \cdots \epsilon_{l_m} \epsilon_{r_1} \cdots \epsilon_{r_n}\\
     & =  &  \epsilon_{r_1} \cdots \epsilon_{r_n} \epsilon_{l_m} \cdots \epsilon_{l_1} \epsilon_{r_1} \cdots \epsilon_{r_n} \epsilon_s [s] = [s^{-1}] \epsilon_{l_1} \cdots \epsilon_{l_m} = \beta.
\end{array}$
\end{center}

The computation is similar to the case  $\alpha\beta\alpha$. From the uniqueness of the writing in the standard form, we get $\{ l_1, \ldots , l_m \} \supseteq \{ r_1, \ldots , r_n \}$. Then $\{ l_1, \ldots , l_m \} = \{ r_1, \ldots , r_n \}$, that is, $\beta = \alpha^*$.
\end{proof}

\begin{exe}
(Disjoint union of groups) Let $G_i$ be groups for $i = 1 , \ldots, n$, and $G = \bigcup\limits_{i = 1}^{n} G_i$ the groupoid constructed by the disjoint union of the $G_i$. Then $S(G) = \bigcup\limits_{i = 1}^{n} S(G_i)$. If the $G_i$'s are finite, then $|G| = \sum\limits_{i = 1}^n |G_i| \Rightarrow |S(G)| = \sum\limits_{i = 1}^n |S(G_i)|$.
\end{exe}

\begin{obs} \label{obsgilbert}
By \cite{gilbert2005actions}, given an ordered groupoid $G$, we can construct the Birget-Rhodes expansion $G^{BR}$ of $G$ as it follows. Consider
\begin{align*}
    F_e(G) = \{ H \subseteq G : e \in H, H \text{ is finite and if } h \in H, \text{ then } r(h) = e \}
\end{align*}
and
\begin{align*}
    F_*(G) = \bigcup_{e \in G_0} F_e(G).
\end{align*}

Then $G^{BR}$ is an ordered groupoid with underlying set
\begin{align*}
    G^{BR} = \{ (U,g) : U \in F_*(G), g \in U \}.
\end{align*}
The composition is given by $(U,g)(V,h) = (U,gh)$ and it is defined when $U = gV$ and $\exists gh$. Also, $(U,g) \leq (V,h)$ if, and only if, $g \leq h$ and $U \supseteq \{ (r(g)|v) : v \in V \}$, where the notation $(e|v)$ means the correstriction of $v$ in $e$.

Since every groupoid can be seen as an ordered groupoid taking equality as the partial order, the construction above can be applied in the nonordered case. Let $G$ be any groupoid. Then $G^{BR}$ is still an ordered groupoid with the same partial product but now ordering $(U,g) \leq (V,h)$ if, and only if, $g = h$ and $U \supseteq V$. With this ordering, we can see that $G^{BR}$ is in fact a locally complete inductive groupoid as defined in \cite{dewolf2018ehresmann}. 

By \cite{dewolf2018ehresmann}, every locally complete inductive groupoid is related with an inverse category by an Ehresmann-Schein-Nambooripad-type theorem. We will denote by $\mathbb{S}(\mathcal{G})$ the inverse category associated with a locally complete inductive groupoid $\mathcal{G}$ via this correspondence.

Finally, it is easy to see that $S(G)$ is not just an inverse semigroupoid, but an inverse category, and it is isomorphic to $\mathbb{S}(G^{BR})$ via $\epsilon_{r_1} \cdots \epsilon_{r_n}[g] \mapsto (\{r_1, \ldots, r_n, r(g), g\},g).$

This argumentation generalizes the results of \cite[Section 2]{kellendonk2004partial} to the case of groupoids. 

Furthermore, every inverse semigroupoid $S$ gives rise to an inverse semigroup as it follows. Consider $S^0 = S \cup \{ 0 \}$, where $0$ is a symbol that is not an element of $S$. We define an operation in $S^0$ by
\begin{align*}
    st = \begin{cases}
      st, & \text{ if } s,t \in S \text{ and } \exists st, \\
      0, & \text{ otherwise.}
    \end{cases}
\end{align*}

In \cite[Proposition 6.16]{lawson2006expansions}, it was defined the Birget-Rhodes expansion $S^{Pr}$ of an inverse semigroup $S$. We observe that $(G^0)^{Pr}$ is isomorphic to $(\mathbb{S}(G^{BR}))^0$ via the isomorphism
\begin{align*}
    (A, s) \in (G^0)^{Pr} \mapsto (A,s) & \in (\mathbb{S}(G^{BR}))^0, \text{ if } s \neq 0, \\
    (\{0\},0) \in (G^0)^{Pr}  & \mapsto 0 \in (\mathbb{S}(G^{BR}))^0.
\end{align*}
\end{obs}

\section{Partial Actions of Groupoids and Actions of Inverse Semigroupoids}

In this section we treat about inverse semigroupoid actions and  partial groupoid  actions. More about  partial groupoid actions can be found in \cite{bagio2012partial}. 

Let $S$ be an inverse semigroupoid with set of idempotent elements $E(S)$ and let $X$ be a non-empty set. An \emph{action of $S$ into $X$} is a pair $\beta = (\{ E_g \}_{g \in S}, \{ \beta_g : E_{g^{-1}} \rightarrow E_g \}_{g \in S})$, where for all $g \in S$, $E_g = E_{gg^*}$ is a subset of $X$ and $\beta_g$ is a bijection. Moreover, the following conditions hold:

\begin{enumerate}
    \item[(A1)] $\beta_{e}$ is the identity map of $E_{e}$, for all $e \in E(S)$ and
    
    \item[(A2)] $\beta_g \circ \beta_h = \beta_{gh}$, for all $g, h \in S$ such that $\exists gh$.
\end{enumerate}

According to \cite{bagio2012partial}, a \emph{partial action $\alpha$ of a groupoid $G$ on a set $X$} is a pair $\alpha = (\{ D_g \}_{g \in G}, \{ \alpha_g : D_{g^{-1}} \rightarrow D_g \}_{g \in G})$, where, for all $g \in G$, $D_{r(g)}$ is a subset of $X$, $D_g$ is a subset of $D_{r(g)}$ and $\alpha_g$ is bijective. Besides that, the following conditions hold:

\begin{enumerate}
    \item[(PA1)] $\alpha_e = I_{D_e}$, for all $e \in G_0$;
    
     \item[(PA2)] $ \alpha_g(D_{g^{-1}} \cap D_h) = D_g \cap D_{gh} $, for all $(g,h) \in G^2$; and
    
    \item[(PA3)] $\alpha_g(\alpha_h(x)) = \alpha_{gh}(x)$, for all $x \in D_{h^{-1}} \cap D_{(gh)^{-1}}$, for all $(g,h) \in G^2$.
\end{enumerate}

The next lemma characterizes partial groupoid actions. Note that this lemma tells us that every partial groupoid action is also a partial groupoid representation (in the sense of \cite{exel1998partial}).

Let $X$ be a non-empty set. Recall that
\begin{align*}
    I(X) = \{ f : A \to B : f \text{ is a bijection}, A,B \subseteq X\}
\end{align*}
is an inverse semigroup regarding to usual 
composition. 

\begin{lema}
\label{lema1} Let $G$ be a groupoid and $X$ a non-empty set. Consider $\alpha : G \rightarrow I(X)$, where we denote $\alpha(s) = \alpha_s$. Then $\alpha$ is a partial groupoid action of $G$ on $X$ if, and only if,

\begin{enumerate}
    \item[\emph{(i)}] $\alpha_s \circ \alpha_t \circ \alpha_{t^{-1}} = \alpha_{st} \circ \alpha_{t^{-1}}$ and
    
    \item[\emph{(ii)}] $\alpha_e = I_{\text{dom}(\alpha_e)}$, for all $e \in G_0$.
\end{enumerate}

In this case, it also holds

\begin{enumerate}
    \item[\emph{(iii)}] $\alpha_{s^{-1}} \circ \alpha_s \circ \alpha_t = \alpha_{s^{-1}} \circ \alpha_{st}$.
\end{enumerate}
\end{lema}

\begin{proof}
$(\Rightarrow):$ Assume that $\alpha = \{ \{ D_s \}_{s \in G}, \{ \alpha_s \}_{s \in G} \}$ is a partial groupoid action. Since (PA1) $\Leftrightarrow$ (ii), it only remains for us to show (i). We have that dom$(\alpha_s \circ \alpha_t \circ \alpha_{t^{-1}}) = \text{dom}(\alpha_s \circ I_{D_t}) = \text{dom}(\alpha_s) \cap D_t = D_t \cap D_{s^{-1}}$ and dom$(\alpha_{st} \circ \alpha_{t^{-1}}) = \alpha_t(D_{t^{-1}} \cap D_{(st)^{-1}})$, for all $(s,t) \in G^2$. Taking $s = t$ and $t = (st)^{-1}$ in (PA2) we obtain $\alpha_t(D_{t^{-1}} \cap D_{(st)^{-1}}) = D_t \cap D_{s^{-1}}$, from where dom$(\alpha_s \circ \alpha_t \circ \alpha_{t^{-1}}) = \text{dom}(\alpha_{st} \circ \alpha_{t^{-1}})$.

Similarly, Im$(\alpha_s \circ \alpha_t \circ \alpha_{t^{-1}}) = \alpha_s(D_t \cap D_{s^{-1}})$ and Im$(\alpha_{st} \circ \alpha_{t^{-1}}) = \alpha_{st}(D_{t^{-1}} \cap D_{(st)^{-1}}) = D_{st} \cap D_s$,

Besides that, from (PA2), we have $\alpha_s(D_t \cap D_{s^{-1}}) = D_s \cap D_{st}$.

Therefore, Im$(\alpha_s \circ \alpha_t \circ \alpha_{t^{-1}}) = \text{Im}(\alpha_{st} \circ \alpha_{t^{-1}})$.

Finally, take $x \in D_t \cap D_{s^{-1}} = \text{dom}(\alpha_s \circ \alpha_t \circ \alpha_{t^{-1}})$. Since $x \in D_t \cap D_{s^{-1}}$, we have that $\alpha_{t^{-1}}(x) = y \in \alpha_{t^{-1}}(D_t \cap D_{s^{-1}}) = D_{t^{-1}} \cap D_{(st)^{-1}}$. Hence $\alpha_s(\alpha_t(\alpha_{t^{-1}}(x))) = \alpha_s(\alpha_t(y)) = \alpha_{st}(y) = \alpha_{st}(\alpha_{t^{-1}}(x))$.

$(\Leftarrow):$ Since $(t, t^{-1}), (t^{-1},t) \in G^2$ for all $t \in G$, we have

\begin{center}$
    \begin{array}{rl}
     \alpha_t \circ \alpha_{t^{-1}} \circ \alpha_t  &  = \alpha_{tt^{-1}} \circ \alpha_t = \alpha_{r(t)} \circ \alpha_t = I_{\text{dom}(\alpha_{r(t)})} \circ \alpha_t = \alpha_t.
     \end{array}$
\end{center}

Similarly, $\alpha_{t^{-1}} \circ \alpha_t \circ \alpha_{t^{-1}} = \alpha_{t^{-1}}$.

The uniqueness of inverse element on inverse semigroupoids assures us that $\alpha_t^{-1} = \alpha_{t^{-1}}$.

Define $D_t = \text{Im}(\alpha_t)$. Then dom$(\alpha_t) = \text{Im}(\alpha_t^{-1}) = \text{Im}(\alpha_{t^{-1}}) = D_{t^{-1}}$.

Therefore, $\alpha_t : D_{t^{-1}} \rightarrow D_t$. Note that given $s \in G$ we have $D_s \subseteq D_{r(s)}$. In fact,

\begin{center}
    $\alpha_{r(s)} \circ \alpha_{s} \circ \alpha_{s^{-1}} = \alpha_{r(s)s} \circ \alpha_{s^{-1}} = \alpha_s \circ \alpha_{s^{-1}} = I_{D_s}$,
\end{center}from where

\begin{center}
    $D_{r(s)} \cap D_s = \text{dom}(\alpha_{r(s)} \circ \alpha_{s} \circ \alpha_{s^{-1}}) = \text{dom}(I_{D_s}) = D_s$.
\end{center}

For all $(s,t) \in G^2$ we have by (i) and (ii) that

\begin{center}
    $\alpha_{t^{-1}} \circ \alpha_{s^{-1}} = \alpha_{t^{-1}} \circ \alpha_{s^{-1}} \circ \alpha_{r(s)} = \alpha_{t^{-1}} \circ \alpha_{s^{-1}} \circ \alpha_{ss^{-1}} =  \alpha_{t^{-1}} \circ \alpha_{s^{-1}} \circ \alpha_s \circ \alpha_{s^{-1}} = \alpha_{t^{-1}s^{-1}} \circ \alpha_{s} \circ \alpha_{s^{-1}}$.
\end{center}

In particular, the domains of these bijections are equal. We have dom$(\alpha_{t^{-1}} \circ \alpha_{s^{-1}}) = \alpha_s(D_{s^{-1}} \cap D_t)$. In fact, given $x \in \text{dom}(\alpha_{t^{-1}} \circ \alpha_{s^{-1}})$ we obtain $x \in \text{dom}(\alpha_{s^{-1}}) = D_s$ and $\alpha_{s^{-1}}(x) \in \text{dom}(\alpha_{t^{-1}}) = D_t$. But $x \in D_s$ implies that $\alpha_{s^{-1}}(x) \in \alpha_{s^{-1}}(D_s) = D_{s^{-1}}$. Thus $\alpha_{s^{-1}}(x) \in D_t \cap D_{s^{-1}}$, from where $x \in \alpha_s(D_{s^{-1}} \cap D_t)$. Hence, dom$(\alpha_{t^{-1}} \circ \alpha_{s^{-1}}) \subseteq \alpha_s(D_{s^{-1}} \cap D_t)$.

For the reverse inclusion, take $x \in \alpha_s(D_{s^{-1}} \cap D_t)$. We need to show that $x \in \text{dom}(\alpha_{t^{-1}} \circ \alpha_{s^{-1}})$, that is, $x \in \text{dom}(\alpha_{s^{-1}}) = D_s$ and that $\alpha_{s^{-1}}(x) \in \text{dom}(\alpha_{t^{-1}}) = D_t$.

Notice that $D_{s^{-1}} \cap D_t \subseteq D_{s^{-1}}$, thus $\alpha_s(D_{s^{-1}} \cap D_t) \subseteq \alpha_s(D_{s^{-1}}) = D_s$. Hence $x \in D_s = \text{dom}(\alpha_{s^{-1}})$ and $\alpha_{s^{-1}}(x)$ is well defined. Since $x \in \alpha_s(D_{s^{-1}} \cap D_t)$ we have

\begin{center}
    $\alpha_{s^{-1}}(x) \in \alpha_{s^{-1}}(\alpha_s(D_{s^{-1}} \cap D_t)) = D_{s^{-1}} \cap D_t \subseteq D_t = \text{dom}(\alpha_{t^{-1}})$.
\end{center} 

On the other hand,

\begin{center}
    dom$(\alpha_{t^{-1}s^{-1}} \circ \alpha_s \circ \alpha_{s^{-1}}) = \text{dom}(\alpha_{t^{-1}s^{-1}} \circ I_{D_s}) = \text{dom}(\alpha_{t^{-1}s^{-1}}) \cap D_s = D_{st} \cap D_s$,
\end{center}from where $\alpha_s(D_{s^{-1}} \cap D_t) = D_{st} \cap D_s$.

Now take $x \in D_{t^{-1}} \cap D_{t^{-1}s^{-1}}$. Since $x \in D_{t^{-1}}$ there exists $y \in D_t$ such that $\alpha_{t^{-1}}(y) = x$. Hence $\alpha_s(\alpha_t(x)) = \alpha_s(\alpha_t(\alpha_{t^{-1}}(y))) = \alpha_{st}(\alpha_{t^{-1}}(y)) = \alpha_{st}(x)$.

Since $x$ is arbitrary, we have that $\alpha_s(\alpha_t(x)) = \alpha_{st}(x), \forall x \in D_{t^{-1}} \cap D_{t^{-1}s^{-1}}$.
Therefore $\alpha = \{ \{ D_s \}_{s \in G}, \{ \alpha_s \}_{s \in G} \}$ is a partial action. Besides that, given $(s,t) \in G^2$, we have

\begin{center}
    $\alpha_{s^{-1}} \circ \alpha_s \circ \alpha_t = (\alpha_{t^{-1}} \circ \alpha_{s^{-1}} \circ \alpha_s)^{-1} = (\alpha_{t^{-1}s^{-1}} \circ \alpha_s)^{-1} = (\alpha_{(st)^{-1}} \circ \alpha_s)^{-1} = \alpha_{s^{-1}} \circ \alpha_{st}$.
\end{center}
\end{proof}

The next lemma gives us a characterization of the idempotents elements in $S(G)$.

\begin{lema}
\label{lemaidem}
Let $G$ be a groupoid. An element $\alpha \in S(G)$ is idempotent if and only if the standard form of $\alpha$ is $\alpha = \epsilon_{r_1} \cdots \epsilon_{r_n}$.
\end{lema}

\begin{proof}
We already know that the elements of the form $\epsilon_{r_1} \ldots \epsilon_{r_n}$, with $r_i \in X_{r_n}$, for all $i = 1, \ldots , n$, are idempotents. 

Now, given $\alpha \in S(G)$, we write $\alpha = \epsilon_{r_1} \ldots \epsilon_{r_n}[s] $ on the standard form, with $r_1, \ldots, r_n \in X_s$. Suppose $\alpha$ idempotent. Then $\exists \alpha^2$ and $\alpha^2 = \alpha$. That is, $d(s) = r(r_1) = r(s)$, and
\begin{align*}
     \alpha^2 & = (\epsilon_{r_1} \ldots \epsilon_{r_n}[s])(\epsilon_{r_1} \ldots \epsilon_{r_n}[s]) = \epsilon_{r_1} \ldots \epsilon_{r_n}\epsilon_{sr_1} \ldots \epsilon_{sr_n}[s]^2 = \epsilon_{r_1} \ldots \epsilon_{r_n}\epsilon_{sr_1} \ldots \epsilon_{sr_n}\epsilon_s[s^2]. 
\end{align*}

From the equality $\alpha^2 = \alpha$ and the uniqueness of the writing on the standard form, we obtain $s^2 = s$, that is, $s = r(s)$, from where $sr_i = r(s)r_i = r(r_i)r_i = r_i$. Then, $\alpha = \epsilon_{r_1} \cdots \epsilon_{r_n}$ as we expected.
\end{proof}

%We can now state the main result of this section.

\begin{teo} \label{teoremaacoes}
Let $X$ be a set and $G$ a groupoid. There is a one-to-one correspondence between

\begin{enumerate}
    \item[\emph{(a)}] partial groupoid actions of $G$ on $X$; and
    
    \item[\emph{(b)}] inverse semigroupoid actions of $S(G)$ on $X$.
\end{enumerate}
\end{teo}

\begin{proof}
(a) $\Rightarrow$ (b): Let $\alpha : G \rightarrow I(X)$ be a partial groupoid action. From Lemma \ref{lema1} and Proposition \ref{prop1}, there is $\beta : S(G) \rightarrow I(X)$ semigroupoid homomorphism such that $\beta([g]) = \alpha(g)$. Therefore, if $\exists gh$, then $\beta_{[g][h]} = \beta([g][h]) = \beta([g])\beta([h]) = \beta_{[g]}\beta_{[h]}$.

The idempotents in $S(G)$ are precisely elements of the form $\alpha = \epsilon_{r_1} \cdots \epsilon_{r_n} \in S(G),$ by Lemma \ref{lemaidem}. Since $\beta$ is a homomorphism, it is enough to show (A1) for elements of the form $\epsilon_r \in S(G)$. 

Define $E_{[g]} = \text{Im}(\beta_{[g]})$. Let $g \in G$. In this way,

\begin{center}
$E_{\epsilon_g} = \text{Im}(\beta_{\epsilon_g}) = \text{Im}(\beta_{[g][g^{-1}]}) = \text{Im}(\beta_{[g]} \circ \beta_{[g^{-1}]}) = \text{Im}(\alpha_g \circ \alpha_{g^{-1}}) = D_g$.    
\end{center}

On the other hand, $E_{[g]} = \text{Im}(\beta_{[g]}) = \text{Im}(\alpha_g) = D_g$, from where $E_{[g]} = E_{[g][g]*} = E_{[g][g^{-1}]} = E_{\epsilon_g}$.

Finally, note that

\begin{center}
     $\beta_{\epsilon_g} = \beta_{[g][g^{-1}]} = \beta_{[g]} \circ \beta_{[g^{-1}]} = \alpha_g \circ \alpha_{g^{-1}} = I_{D_g} = I_{E_{\epsilon_g}}$,
\end{center}
what guarantees that $\beta$ is an inverse semigroupoid action of $S(G)$ on $X$.

(b) $\Rightarrow$ (a): Let $\beta$ be an action of $S(G)$ on $X$. Define $\alpha(g) = \beta([g])$, for all $g \in G$. Note that if $(s,t) \in G^2$ and $e \in G_0$ then
\begin{align*}
    \alpha(s) \circ \alpha(t) \circ \alpha(t^{-1}) & = \beta([s]) \circ \beta([t]) \beta([t^{-1}])  = \beta([s][t][t^{-1}]) = \beta([st][t^{-1}]) = \beta([st]) \circ \beta([t^{-1}])\\
    & = \alpha(st) \circ \alpha(t^{-1}).
\end{align*}

Besides that, $r(e) = d(e) = e = e^2$, from where $e$ is idempotent and $[e] = \epsilon_e$ is as well. Hence,
\begin{align*}
     \alpha(e)& = \beta([e]) = I_{E_{[e]}} = I_{\text{dom}(\beta([e]))} = I_{\text{dom}(\alpha(e))}.
\end{align*}

That concludes the proof by Lemma \ref{lema1}.
\end{proof}

\section{Crossed Products}

Let us remember some results of partial groupoid actions theory. 

\begin{defi} \label{defskew} \cite[Section 3]{bagio2010partial}
Let $R$ be an algebra, $G$ a groupoid and $\alpha = (\{ D_g \}_{g \in G}, \{ \alpha_g \}_{g \in G})$ a partial action of $G$ on $R$. We define the algebraic crossed product $R \ltimes_{\alpha}^a G$ by
\begin{center}
    $R \ltimes_{\alpha}^a G = \left \{ \sum\limits_{g \in G}^{\text{finite}} a_g\delta_g \right \} = \bigoplus\limits_{g \in G} D_g\delta_g$
\end{center}
where $\delta_g$ are symbols that represent the position of the element $a_g$ in the sum. The addition is usual and the product is given by

\begin{center}
$(a_g\delta_g)(b_h\delta_h) = 
     \begin{cases}
       \alpha_g(\alpha_{g^{-1}}(a_g)b_h)\delta_{gh}, &\quad\text{if } (g,h) \in G^2. \\
       0, &\quad\text{otherwise}
     \end{cases}    $
\end{center}
and linearly extended.
\end{defi}

The next result gives a condition to the associativity of the algebraic crossed product. 

\begin{propo} \label{teoassoc}
Let $R$ be an algebra, $G$ a groupoid and $\alpha = (\{ D_g \}_{g \in G},$ $\{ \alpha_g \}_{g \in G})$ a partial action of $G$ on $R$. If $R$ is semiprime, then the algebraic crossed product $R \ltimes_{\alpha}^a G$ is associative.
\end{propo}

The proof can be found in \cite[Proposition 3.1]{bagio2010partial}, adding that every groupoid can be seen as an ordered groupoid regarding to equality as partial order.

We are interested in defining the algebraic crossed product of a action of an inverse semigroupoid on an algebra. Exel and Vieira made in \cite{exel2010actions} the first study in the case of groups.

\begin{propo} \label{prop9}
Let $\beta$ be an action of an inverse semigroupoid $S$ on an algebra $R$. If $\exists st$ then

\begin{enumerate}
    \item[\emph{(i)}] $\beta_{s^{-1}} = \beta_s^{-1}$
    
    \item[\emph{(ii)}] $\beta_s(E_t) = \beta(E_t \cap E_{s^{-1}}) = E_{st}$
    
    \item[\emph{(iii)}] $E_{st} \subseteq E_s$
\end{enumerate}
\end{propo}

\begin{proof}
(i): Since $ss^{-1} \in E(S)$, we have $I_{E_s} = I_{E_{ss^{-1}}} = \beta_{ss^{-1}} = \beta_s \circ \beta_{s^{-1}}$, that is, $\beta_{s^{-1}} = \beta_s^{-1}$.

\noindent (ii): Just observe that $\beta_s(E_t) = \beta_s(E_t \cap E_{s^{-1}}) = \beta_{s^{-1}}^{-1}(E_t \cap E_{s^{-1}}) = \text{dom}(\beta_{t^{-1}} \circ \beta_{s^{-1}}) = \text{dom}(\beta_{(st)^{-1}}) = E_{st}$.

\noindent (iii): Follows from (ii).
\end{proof}

\begin{defi} \label{defl}
Let $\beta = (\{ E_s \}_{s \in S}, \{ \beta_s \}_{s \in S})$ be an action of an inverse semigroupoid $S$ on an algebra $R$. We define the crossed product by
%Consider $E_s = E_{ss^*} \vartriangleleft R$ and $\beta_s$ an isomorphism for all $s \in S$
\begin{center}
    $L = \left \{ \sum\limits_{s \in S}^{\text{finite}} a_s\delta_s \right \} = \bigoplus\limits_{s \in S} E_s\delta_s$.
\end{center}

The addition is usual and the product is given by

\begin{center}
$(a_s\delta_s)(b_t\delta_t) = 
     \begin{cases}
       \beta_s(\beta_{s^*}(a_s)b_t)\delta_{st}, &\quad\text{if } \exists st. \\
       0, &\quad\text{otherwise}
     \end{cases}    $
\end{center}
and linearly extended.

\end{defi}

\begin{propo}
Let $R$ be an algebra, $S$ an inverse semigroupoid and $\beta = (\{ E_s \}_{s \in S}, \{ \beta_s \}_{s \in S})$ an action of $S$ on $R$. If $R$ is semiprime, then the crossed product $L$ defined in \ref{defl} is associative.
\end{propo}

\begin{proof}
The proof follows similarly the arguments given in \cite[Proposition 3.1]{bagio2010partial}.
\end{proof}

Finally we can define the algebraic crossed product of the inverse semigroupoid $S$ on the semiprime algebra $R$. Recall from \cite{liu2016free} that every inverse semigroupoid has a natural partial order relation: given $s, t$ in an inverse semigroupoid $S$, we have $$r \leq t \Leftrightarrow \exists i \in E(S), \text{ such that } \exists it \text{ and } r = it.$$

\begin{defi}
Let $R$ be a semiprime algebra, $S$ an inverse semigroupoid and $\beta = (\{ E_s \}_{s \in S}, \{ \beta_s \}_{s \in S})$ an action of $S$ on $R$. Define $N = \langle a\delta_r - a\delta_t : a \in E_r, r \leq t \rangle$. The algebraic crossed product of $S$ on $R$ by $\beta$ is

\begin{center}
    $R \ltimes_{\beta}^a S = L/N$.
\end{center}
\end{defi}

\begin{obs}
The algebraic crossed product of an inverse semigroupoid $S$ on a semiprime algebra $R$ by an action $\beta$ is well defined.  If $r \leq t$, then $E_r \subseteq E_t$, by Proposition \ref{prop9}. So we can write $a\delta_t$ even though $a \in E_r$.
\end{obs}

The natural partial order of an inverse semigroupoid plays an important role in the crossed product. Since we want to work with $S(G)$, we must understand its partial order.

\begin{lema}
Let $G$ be a groupoid. If $\alpha, \beta \in S(G)$ are such that $\alpha = \epsilon_{r_1} \cdots \epsilon_{r_n} [s]$, $\beta = \epsilon_{\ell_1} \cdots \epsilon_{\ell_m} [t]$ and $\alpha \leq \beta$, then $s = t$ and $\{ \ell_1 , \ldots , \ell_m \} \subseteq \{ r_1 , \ldots , r_n \}$.
\end{lema}

\begin{proof}
Since $\alpha \leq \beta$, there is an idempotent $\gamma \in S(G)$ such that $\exists \beta\gamma$ and $\alpha = \beta\gamma$. From Lemma \ref{lemaidem} we know that $\gamma = \epsilon_{p_1} \cdots \epsilon_{p_k}$ on the standard form (where $p_1, \ldots , p_k \in X_{p_1}$). Then,

\begin{center}
    $\epsilon_{r_1} \cdots \epsilon_{r_n} [s] = \alpha = \beta\gamma = \epsilon_{\ell_1} \cdots \epsilon_{\ell_m} [t] \epsilon_{p_1} \cdots \epsilon_{p_k} = \epsilon_{\ell_1} \cdots \epsilon_{\ell_m} \epsilon_{tp_1} \cdots \epsilon_{tp_k}[t]$.
\end{center}

The result follows from the uniqueness of decomposition on standard form.
\end{proof}

Notice that this ordering coincides with the ordering in the Birget-Rhodes expansion $G^{BR}$ that was presented in Remark \ref{obsgilbert}.

\begin{lema} \label{lema5}
Let $G$ be a groupoid. Given $r_1, \ldots , r_n, g, h \in G$, $r_i \in X_g$ for all $i = 1 , \ldots , n$, $d(g) = r(h)$, we have

\begin{enumerate}
    \item[\emph{(i)}] $E_{[g][h]} = E_{[g]} \cap E_{[gh]}$
    
    \item[\emph{(ii)}] $E_{\epsilon_{r_1} \cdots \epsilon_{r_n}[g]} \subseteq E_{[g]}$
\end{enumerate}
\end{lema}

\begin{proof}
(i): $
\begin{array}{r}
     E_{[g][h]}  =  E_{[g][g^{-1}][g][h]} =  E_{[g][g^{-1}][gh]} =  \beta_{[g]}(E_{[g^{-1}][gh]}) =  \beta_{[g]} \circ \beta_{[g^{-1}]}(E_{[gh]}) =  E_{[g]} \cap E_{[gh]}.
\end{array}$

(ii): Follows from (i).
\end{proof}

The last lemma before the principal result of this section will give us two important relations on the crossed product.

\begin{lema}
Let $R$ be a semiprime algebra, $G$ a groupoid and $\beta = (\{ E_s \}_{s \in S(G)}, \{ \beta_s \}_{s \in S(G)})$ an action of $S(G)$ on $R$. For $r_1, \ldots , r_n, g, h \in G$, such that $\exists gh$, $r(r_i) = r(g)$, for all $i = 1, \ldots , n$, the equalities below hold in $R \ltimes_\beta^a S(G)$:

\begin{enumerate}
    \item[\emph{(i)}] $\overline{a\delta_{[g][h]}} = \overline{a\delta_{[gh]}}$, for all $a \in E_{[g][h]}$
    
    \item[\emph{(ii)}] $\overline{a\delta_{\epsilon_{r_1} \cdots \epsilon_{r_n}[g]}} = \overline{a\delta_{[g]}}$, for all 
\end{enumerate}

\end{lema}

\begin{proof}
(i): Since $E_{[g][h]} \subseteq E_{[gh]}$ from Lemma \ref{lema5}, we can write $a\delta_{[gh]}$. Just note that 

\begin{center}
$[g][h] \leq [gh]$, since $[g][h] = [g][h][h^{-1}][h] = [gh][h^{-1}][h] = [gh]\epsilon_{h^{-1}}$    
\end{center}
and $\epsilon_{h^{-1}}$ is idempotent. Then, $a\delta_{[g][h]} - a\delta_{[gh]} \in N$.

(ii): Since $E_{\epsilon_{r_1} \cdots \epsilon_{r_n}[g]} \subseteq E_{[g]}$ we can write $a\delta_{[g]}$. Now note that $\epsilon_{r_1} \cdots \epsilon_{r_n}[g] = [g]\epsilon_{g^{-1}r_1} \cdots \epsilon_{g^{-1}r_n}$ and $\epsilon_{g^{-1}r_1} \cdots \epsilon_{g^{-1}r_n} \in E(S(G))$, from where $\epsilon_{r_1} \cdots \epsilon_{r_n}[g] \leq [g]$. Hence $a\delta_{\epsilon_{r_1} \cdots \epsilon_{r_n}[g]} - a\delta_{[g]} \in N$.
\end{proof}

\begin{teo}
Let $G$ be a groupoid, $R$ a semiprime algebra and $\alpha$ a partial action of $G$ on $R$. Let also $\beta$ be the partial action of $S(G)$ on $A$ associated with $\alpha$ as in the Theorem \ref{teoremaacoes}. So $R \ltimes_{\alpha}^a G \cong R \ltimes_{\beta}^a S(G)$.
\end{teo}

\begin{proof}
Define 

\begin{minipage}[b]{0.4\linewidth}
\begin{gather*}
    \varphi: R \ltimes_{\alpha}^a G \rightarrow R \ltimes_{\beta}^a S(G) \text{\quad \quad \quad  and} \\
    a\delta_g \mapsto \overline{a\delta_{[g]}}
\end{gather*}
\end{minipage}
\begin{minipage}[b]{0.3\linewidth}
\begin{gather*}
    \psi: L \rightarrow R \ltimes_{\alpha}^a G \\
    a\delta_\gamma \mapsto a\delta_{\partial(\gamma)}.
\end{gather*}
\end{minipage}

It is easy to see that $\varphi$ and $\psi$ are algebra homomorphisms and that $N \subseteq \ker \psi$. Therefore exists a unique homomorphism $\overline{\psi} : R \ltimes_{\beta}^a S(G) \rightarrow R \ltimes_{\alpha}^a G$ such that $\overline{\psi}(\overline{a\delta_{\gamma}}) = a\delta_{\partial(\gamma)}$. Notice that $\varphi$ and $\overline{\psi}$ are inverses, consequently isomorphisms.
\end{proof}

\section{Partial Groupoid Representations and Inverse Semigroupoid Representations}

In this section we define inverse semigroupoid representations on Hilbert spaces and partial groupoid representations on Hilbert spaces.

\begin{defi}
Let $S$ be an inverse semigroupoid and $H$ a Hilbert space. A representation of $S$ on $H$ is a map $\pi : S \rightarrow \mathcal{B}(H)$ such that

\begin{enumerate}
    \item[(R1)] if $\exists \alpha \beta$, then $\pi(\alpha\beta) = \pi(\alpha)\pi(\beta)$;
    
    \item[(R2)] $\pi(\alpha^*) = \pi(\alpha)^*$; and
    
    \item[(R3)] $\pi(\alpha\alpha^*)\pi(\alpha)$.
\end{enumerate}
\end{defi}

\begin{defi}
Let $G$ be a groupoid and $H$ a Hilbert space. A partial representation of $G$ on $H$ is a map $\pi : G \rightarrow \mathcal{B}(H)$ such that

\begin{enumerate}
    \item[(PR1)] $\pi(s)\pi(t)\pi(t^{-1}) = \pi(st)\pi(t^{-1})$, for all $(s,t) \in G^2$;
    
    \item[(PR2)] $\pi(s^{-1}) = \pi(s)^*$; and
    
    \item[(PR3)] $\pi(r(s))\pi(s) = \pi(s)$, for all $e \in G_0$.
\end{enumerate}

In this case, it also holds

\begin{center}
    $\pi(s^{-1})\pi(s)\pi(t) = \pi(s^{-1})\pi(st)$.
\end{center}

\end{defi}

\begin{propo} \label{proprep}
Let $G$ be an groupoid and $H$ a Hilbert space. There is a one-to-one correspondence between

\begin{enumerate}
    \item[\emph{(a)}] partial groupoid representations of $G$ on $H$; and
    
    \item[\emph{(b)}] inverse semigroupoid representations of $S(G)$ on $H$.
\end{enumerate}
\end{propo}

\begin{proof}
A partial groupoid representation of $G$ on $H$ holds the conditions of Proposition \ref{prop1}. Hence there is a semigroupoid homomorphism $\overline{\pi}: S(G) \rightarrow \mathcal{B}(H)$ such that $\overline{\pi}([t]) = \pi(t)$ for all $t \in G$. This already gives us (R1) and (R3). 

To show (R2), let $\alpha = \epsilon_{r_1}\epsilon_{r_2} \cdots \epsilon_{r_n}[s] \in S(G)$ be such that $r_1, \ldots , r_n \in X_s$, as in Propositon \ref{propdec}. We know that we can write

\begin{center}
$\alpha^* = [s^{-1}]\epsilon_{r_n} \cdots \epsilon_{r_1}$.    
\end{center}

Therefore,
\begin{align*}
     \overline{\pi}(\alpha^*) & = \overline{\pi}([s^{-1}]\epsilon_{r_n} \cdots \epsilon_{r_1}) = \overline{\pi}([s^{-1}][r_n][r_n^{-1}] \cdots [r_1][r_1^{-1}]) = \overline{\pi}([s^{-1}])\overline{\pi}([r_n])\overline{\pi}([r_n^{-1}]) \cdots \overline{\pi}([r_1])\overline{\pi}([r_1^{-1}]) \\
     & = \pi(s^{-1})\pi(r_n)\pi(r_n^{-1}) \cdots \pi(r_1)\pi(r_1^{-1}) = \pi(s)^*\pi(r_n)\pi(r_n)^* \cdots \pi(r_1)\pi(r_1)^*\\
     & = (\pi(r_1)\pi(r_1)^* \cdots \pi(r_n)\pi(r_n)^*\pi(s))^* = (\pi(r_1)\pi(r_1^{-1}) \cdots \pi(r_n)\pi(r_n^{-1})\pi(s))^* \\
     &  = (\overline{\pi}([r_1])\overline{\pi}([r_1^{-1}]) \cdots \overline{\pi}([r_n])\overline{\pi}([r_n^{-1}])\overline{\pi}([s]))^* = \overline{\pi}([r_1][r_1^{-1}] \cdots [r_n][r_n^{-1}][s])^* = \overline{\pi}(\epsilon_{r_1} \cdots \epsilon_{r_n}[s])^* \\
     & = \overline{\pi}(\alpha)^*,
\end{align*}
which is what we wanted to demonstrate. 

On the other hand, given a semigroupoid representation $\rho$ of $S(G)$ on $H$, consider the map
\begin{gather*}
		\pi : G \rightarrow \mathcal{B}(H) \\
		t \mapsto \pi(t) = \rho([t]).
\end{gather*}

We will show that $\pi$ is a partial groupoid representation of $G$ on $H$. Take $(s,t) \in G^2$.

(PR1):\begin{align*}
     \pi(s)\pi(t)\pi(t^{-1}) & = \rho([s])\rho([t])\rho([t^{-1}])  = \rho([s][t][t^{-1}])  = \rho([st][t^{-1}]) = \rho([st])\rho([t^{-1}]) = \pi(st)\pi(t^{-1}). 
\end{align*}

(PR2): $\pi(t^{-1}) = \rho([t^{-1}]) = \rho([t]^*) = \rho([t])^* = \pi(t)^*$.

(PR3): $\pi(r(s))\pi(s) = \pi(ss^{-1})\pi(s) = \pi(s)\pi(s^{-1})\pi(s) = \rho([s])\rho([s^{-1}])\rho([s]) = \rho([s][s]^{-1})\rho([s]) $ $= \rho([s]) = \pi(s)$.

Therefore $\pi$ is a partial groupoid representation of $G$ on $H$.
\end{proof}

\section{Exel's Partial Groupoid C*-Algebra}

The main goal of this section is to generalize the definition of partial group $C^*$-algebra to obtain the purely algebraic definition of partial groupoid $C^*$-algebra. 

It is proved in \cite{exel2010actions} that if the algebra $R$ on Definition \ref{defskew} is a $C^*$-algebra, then the algebraic crossed product $R \ltimes_{\alpha}^a G$ admits an enveloping $C^*$-algebra with the relations

\begin{center}
    $(a_g\delta_g)^* = \alpha_{g^{-1}}(a_g^*)\delta_{g^{-1}}$
\end{center}
and

\begin{center}
    $ \left\Vert \sum\limits_{g \in G} a_g\delta_g \right\Vert= \sum\limits_{g \in G} \Vert a_g \Vert $.
\end{center}

This also holds on the cases of groupoids and inverse semigroupoids and can be easily proven in an analogous way.

\begin{obs}
One can show that if $R$ is a $C^*$-algebra and $\alpha$ is a partial groupoid action of a groupoid $G$ on $R$, we have $R \ltimes_{\alpha} G \cong R \ltimes_{\beta} S(G)$, where $R \ltimes_{\alpha} G$ denotes the crossed product of $R$ by $G$, that is, the enveloping $C^*$-algebra of $R \ltimes_{\alpha}^a G$ and $\beta$ is the inverse semigroupoid action associated to $\alpha$ as in Theorem \ref{teoremaacoes}. Similarly $R \ltimes_{\beta} S(G) = C_e^*(R \ltimes_{\beta}^a S(G))$.
\end{obs}

\begin{defi}
Define an auxiliar $C^*$-algebra $R$ by generators and relations. The generators are the symbols $P_E$, where $E \subseteq G$ is a finite subset of the groupoid $G$. The relations are

\begin{enumerate}
    \item[(i)] $P_E^* = P_E$; and
    
    \item[(ii)] $P_EP_F = 
     \begin{cases}
       P_{E \cup F}, &\quad\text{if } E \cup F \subseteq X_g, \text{ for some } g \in G. \\
       0, &\quad\text{otherwise.}
     \end{cases}    $
\end{enumerate}

\end{defi}

\begin{propo}
$R$ is an abelian $C^*$-algebra with $P_{\emptyset}$ as unity. Besides that, every element of $R$ is a projection.
\end{propo}
\begin{proof}
    It follows from the following properties of set theory: $E \cup F = F \cup E$, $E \cup \emptyset = E = \emptyset \cup E$ and $E \cup E = E$.
\end{proof}

The above proposition tells us that if $E \subseteq G$ is such that $E \nsubseteq X_g$ for all $g \in G$, then $P_E = 0$.

Given $t \in G$, consider the map $\alpha_t : R \rightarrow R$ defined by 

\begin{center}
$\alpha_t(P_E) = 
     \begin{cases}
       P_{tE}, &\quad\text{if } \exists tx, \text{ for all } x \in E. \\
       0, &\quad\text{otherwise,} 
     \end{cases}     $
\end{center}
that is,
\begin{center}
$\alpha_t(P_E) = 
     \begin{cases}
       P_{tE}, &\quad\text{if } E \subseteq X_{t^{-1}}. \\
       0, &\quad\text{otherwise,} 
     \end{cases}     $
\end{center}

If we define $D_t = \overline{\text{span}}\{P_E : t, r(t) \in E \subseteq X_t \}$, we have that $D_t \vartriangleleft D_{r(t)} \vartriangleleft R$, for all $t \in G$, and the restriction of every $\alpha_t$ to $D_{t^{-1}}$ gives us a partial action of $G$ on $R$. Therefore we can define the algebraic crossed product $R \ltimes_{\alpha}^a G$. That gives us the following definition:

\begin{defi}
The Exel's partial groupoid $C^*$-algebra $C_p^*(G)$ is the enveloping $C^*$-algebra of $R \ltimes_{\alpha}^a G$, that is, $C_p^*(G) = R \ltimes_{\alpha} G$.

\end{defi}

The next proposition follows from Proposition \ref{teoassoc} and from the fact that an algebra $A$ is semiprime if, and only if, its ideals are idempotent.

\begin{propo}
$C_p^*(G)$ is an associative $C^*$-algebra.
\end{propo}

\begin{defi}
A $C^*$-algebra representation of a $C^*$-algebra $A$ on a Hilbert space $H$ is a $^*$-homomor-phism of $\mathbb{C}$-algebras $\phi : A \rightarrow \mathcal{B}(H)$.
\end{defi}

\begin{defi}
Let $G$ be a groupoid, $R$ a $C^*$-algebra and $\alpha$ a partial action of $G$ on $R$. We call the triple $(R,G,\alpha )$ a partial groupoid $C^*$-dynamic system.

A covariant representation of $(R,G,\alpha)$ is a triple $(\pi,u,H)$, where $\pi : R \rightarrow \mathcal{B}(H)$ is a representation of $R$ on a Hilbert space $H$ and $u : G \rightarrow \mathcal{B}(H)$ is a map given by $u(g) = u_g$, where $u_g$ is a partial isometry such that

\begin{enumerate}
    \item[(CR1)] $u_g\pi(x)u_{g^{-1}} = \pi(\alpha_g(x)),$ for all $x \in D_{g^{-1}}$
    
    \item[(CR2)] $\pi(x)u_gu_h = 
     \begin{cases}
       \pi(x)u_{gh}, &\quad\text{if } (g,h) \in G^2  \text{ and } x \in D_g \cap D_{gh}. \\
       0, &\quad\text{if } \nexists gh.
     \end{cases}     $
     
     \item[(CR3)] $u_g^* = u_{g^{-1}}$.
\end{enumerate}
\end{defi}

The definition of covariant representation give us the following relations, listed in the lemma below. 

\begin{lema} \label{lemaaux}
Let $(\pi, u, H)$ be a covariant representation of $(R,G, \alpha)$. If $x \in D_g$, then:

\begin{enumerate}
    \item[\emph{(i)}] $\pi(x)u_gu_{g^{-1}} = \pi(x)$
    
    \item[\emph{(ii)}] $\pi(x) = u_gu_{g^{-1}}\pi(x)$
\end{enumerate}
\end{lema}

\begin{proof}
Since $D_g = D_g \cap D_{r(g)} = D_g \cap D_{gg^{-1}}$, we obtain from (CR2) that $\pi(x)u_gu_{g^{-1}} = \pi(x)u_{r(g)}$.

Moreover, from (CR1) we have $\pi(x) = u_gu_{g^{-1}}\pi(x)u_gu_{g^{-1}} = u_gu_{g^{-1}}\pi(x)u_{r(g)}$.

Note that if $x \in D_g = D_g \cap D_g = D_g \cap D_{gd(g)}$, then by (CR2) $\pi(x)u_gu_{d(g)} = \pi(x)u_g$. From that and (CR1) we get $\pi(\alpha_{g^{-1}}(x)) = u_{g^{-1}}\pi(x)u_gu_{d(g)} = u_{g^{-1}}\pi(x)u_g = \pi(\alpha_{g^{-1}}(x))$. But this is the same as saying $\pi(x)u_{d(g)} = \pi(x)$, for all $x \in D_{g^{-1}}$, since $\alpha_{g^{-1}}$ is an isomorphism. Therefore, switching $g$ for $g^{-1}$, we get from that and from the fact that $d(g^{-1}) = r(g)$, that $\pi(x)u_{r(g)} = \pi(x)$, from where follows the result.
\end{proof}

\begin{defi}
Let $(\pi,u,H)$ be a covariant representation of $(R,G,\alpha)$. Define $\pi \times u : R \ltimes_{\alpha} G \rightarrow \mathcal{B}(H)$ by $(\pi \times u)(a_g\delta_g) = \pi(a_g)u_g$ and linearly extended.
\end{defi}

\begin{lema}
$\pi \times u$ is a $^*$-homomorphism.
\end{lema}

\begin{proof}
Let $a_g\delta_g , b_h\delta_h \in R \ltimes_{\alpha} G$. We will show that $(\pi \times u)((a_g\delta_g)(b_h\delta_h)) = (\pi \times u)(a_g\delta_g)(\pi \times u)(b_h\delta_h)$ and that $(\pi \times u)((a_g\delta_g)^*) = (\pi \times u)(a_g\delta_g)^*$. That will be enough, because it will only take us to extend the results linearly to obtain the lemma. 

First note that $(\pi \times u)((a_g\delta_g)(b_h\delta_h)) = \pi(a_g)u_g\pi(b_h)u_h$. But since $a_g \in D_g$, we can use Lemma \ref{lemaaux}(ii) to guarantee that $\pi(a_g) = u_gu_{g^{-1}}\pi(a_g)$. Therefore,

\begin{center}
     $(\pi \times u)((a_g\delta_g)(b_h\delta_h)) = u_gu_{g^{-1}}\pi(a_g)u_g\pi(b_h)u_h$.
\end{center}

Note now that from (CR1)

\begin{center}
     $(\pi \times u)((a_g\delta_g)(b_h\delta_h)) = u_g\pi(\alpha_{g^{-1}}(a_g))\pi(b_h)u_h = u_g\pi(\alpha_{g^{-1}}(a_g)b_h)u_h$. 
\end{center}

Since $b_h \in D_h$, then $\alpha_{g^{-1}}(a_g) \in D_{g^{-1}}$. Since $D_h$ and $D_{g^{-1}}$ are ideals, we have that $\alpha_{g^{-1}}(a_g)b_h $ $\in D_{g^{-1}} \cap D_h$. Thus

\begin{center}
     $(\pi \times u)((a_g\delta_g)(b_h\delta_h)) = u_g\pi(\alpha_{g^{-1}}(a_g)b_h)u_{g^{-1}}u_gu_h = \pi(\alpha_g(\alpha_{g^{-1}}(a_g)b_h))u_gu_h$.
\end{center}

On the other hand,

\begin{center}
     $\pi(\alpha_g(\alpha_{g^{-1}}(a_g)b_h))u_gu_h = \begin{cases}
       \pi(\alpha_g(\alpha_{g^{-1}}(a_g)b_h))u_{gh}, &\quad\text{if } (g,h) \in G^2. \\
       0, &\quad\text{otherwise.}
     \end{cases}$
\end{center}

The equality above holds because since $\alpha_{g^{-1}}(a_g)b_h \in D_{g^{-1}} \cap D_h$, we have that if $\exists gh$, then $\alpha_g(\alpha_{g^{-1}}(a_g)b_h) \in \alpha_g(D_{g^{-1}} \cap D_h) = D_g \cap D_{gh}$. So we can apply (CR2).

Now note that if $\nexists gh$, then $(a_g\delta_g)(b_h\delta_h) = 0$, from where

\begin{center}
    $(\pi \times u)((a_g\delta_g)(b_h\delta_h)) = (\pi \times u)(0) = (\pi \times u)(0\delta_g) = \pi(0)u_g = 0$.
\end{center}

That is, if $\nexists gh$, then $\pi \times u$ is multiplicative. In addition, if $\exists gh$, then 

\begin{center}
    $(\pi \times u)((a_g\delta_g)(b_h\delta_h)) = \pi(\alpha_g(\alpha_{g^{-1}}(a_g)b_h))u_{gh} = (\pi \times u)(\alpha_g(\alpha_{g^{-1}}(a_g)b_h) \delta_{gh}) = (\pi \times u)((a_g\delta_g)(b_h\delta_h))$.
\end{center}

Hence, $\pi \times u$ is also multiplicative when $\exists gh$, from where we can conclude that $\pi$ is a homomorphism since it is linear by definition. Moreover,
\begin{align*}
     (\pi \times u)(a_g\delta_g)^* & = (\pi(a_g)u_g)^* = (u_g)^*\pi(a_g)^* = u_{g^{-1}}\pi(a_g^*) = u_{g^{-1}}\pi(a_g^*)u_gu_{g^{-1}}\\
     & = \pi(\alpha_{g^{-1}}(a_g^*))u_{g^{-1}} = (\pi \times u) (\alpha_{g^{-1}}(a_g^*)\delta_{g^{-1}}) = (\pi \times u)((a_g\delta_g)^*),
\end{align*}

that is, $\pi \times u$ is a $^*$-homomorphism.
\end{proof}

We can now state the main result of this work.

\begin{teo}
Let $G$ be a groupoid and $H$ a Hilbert space. There is a one-to-one correspondence between
\begin{enumerate}
    \item[\emph{(a)}] partial groupoid representations of $G$ on $H$;
    
    \item[\emph{(b)}] inverse semigroupoid representations of $S(G)$ on $H$;
    
    \item[\emph{(c)}] $C^*$-algebra representations of $C_p^*(G)$ on $H$.
\end{enumerate}
\end{teo}

\begin{proof}
We already know from Proposition \ref{proprep} that (a) $\Leftrightarrow$ (b). We will prove (b) $\Rightarrow$ (c) $\Rightarrow$ (a). 

(b) $\Rightarrow$ (c): Let $E = \{ r_1 , \ldots , r_n \} \subseteq G$ be a finite subset of $G$ and $\pi : S(G) \rightarrow \mathcal{B}(H)$ be a representation of $S(G)$ on $H$. Define 

\begin{center}
    $Q_E = \begin{cases}
    \pi(\epsilon_{r_1} \cdots \epsilon_{r_n}), &\quad\text{if } E \subseteq X_g, \text{ for some } g \in G. \\
       0, &\quad\text{otherwise.}
    \end{cases}$
\end{center}

Note that $Q_E$ is well defined, because if $E \subseteq X_g$ we have that $r(r_i) = r(g)$, for all $i = 1, \ldots , n$. Therefore $d(r_i^{-1}) = r(r_i) = r(r_{i+1})$, so $\exists r_i^{-1}r_{i+1}$, for all $i = 1 , \ldots , n-1$, from where $\exists \epsilon_{r_i}\epsilon_{r_{i +1}}$, for all $i = 1 , \ldots , n-1$.

It is obvious that $Q_E^2 = Q_E = Q_E^*$. In addition, 
\begin{center}
    $Q_EQ_F = \begin{cases}
    Q_{E \cup F}, &\quad\text{if } E \cup F \subseteq X_g, \text{ for some } g \in G. \\
       0, &\quad\text{otherwise.}
    \end{cases}$
\end{center}

Defining $\rho : R \rightarrow \mathcal{B}(H)$ by $\rho(P_E) = Q_E$, we get a $C^*$-algebra representation of $R$ on $H$. Define also $u : G \rightarrow \mathcal{B}(H)$ by $u(g) = u_g = \pi([g])$. It is obvious that $(\rho, u, H)$ is a covariant representation of $(R,G,\alpha)$. Hence we can consider the $^*$-homomorphism $\rho \times u$, that is a $C^*$-algebra representation of $C_p^*(G)$ on $H$.

(c) $\Rightarrow$ (a): Let $\phi : C_p^*(G) \rightarrow \mathcal{B}(H)$ be a $C^*$-algebra representation. Consider the elements of the form $a_t = P_{ \{ r(t), t \}}\delta_t \in C_p^*(G)$. It is easy to see that if $(s,t) \in G^2$, then $a_sa_ta_{t^{-1}} = a_{st}a_{t^{-1}}$. It is also clear that $a_t^* = a_{t^{-1}}$. Moreover, $a_e = 1_{D_e}\delta_e$, for all $e \in G_0$. So if we define $\pi : G \rightarrow \mathcal{B}(H)$ by $\pi(t) = \phi(a_t)$ we get a partial groupoid representation of $G$ on $H$.
\end{proof}

\renewcommand{\bibname}{Bibliography}

\end{document}